\documentclass[10pt]{amsart}
\usepackage{amssymb}

\usepackage{graphicx}

\theoremstyle{plain}
\newtheorem{thm}{Theorem}
\newtheorem{prop}[thm]{Proposition}
\newtheorem{lm}[thm]{Lemma}

\newtheorem{cor}[thm]{Corollary}

\theoremstyle{definition}

\newtheorem{exm}[thm]{Example}
\newtheorem{df}[thm]{Definition}

\theoremstyle{remark}
\newtheorem{remark}[thm]{Remark}

\renewcommand{\leq}{\leqslant}

\newcommand{\rto}{\rightarrow}

\newcommand{\id}{\operatorname{Id}}

\DeclareMathOperator{\Cld}{Cld}
\DeclareMathOperator{\Pow}{Pow}

\newcommand{\F}{\mathcal F}

\newcommand{\op}{\operatorname}

\newcommand{\qedc}{{\qed}~{\rm Claim~{\theclaim}.}}

\begin{document}

\title{Algebraic convex geometries revisited}
\author {K. Adaricheva}
\address{School of Science and Technology, Nazarbayev University,
53 Kabanbay Batyr ave., Astana, Republic of Kazakhstan, 010000}
\email{kira.adaricheva@nu.edu.kz}
\address{Department of Mathematical Sciences, Yeshiva University,
New York, NY 10016, USA}
\email{adariche@yu.edu}

\thanks{The work on this paper was partially supported by grant of Nazarbayev University N 13/42.}

\begin{abstract}
Representation of convex geometry as an appropriate join of compatible total orderings of the base set can be achieved, when closure operator of convex geometry is algebraic, or finitary. This bears to the finite case proved by P.~H.~Edelman and R.~E.~Jamison to the greater extent than was thought before.
\end{abstract}

\maketitle

\section{Introduction}
This paper is stimulated by the work of the author on the chapters for the second volume of the Special topics and Applications series following new edition of G.~Gr\"atzer's book on lattice theory \cite{GLTF}, the first volume \cite{LTSP1} of the series  appeared not long ago. One of the chapters in the upcoming second volume is devoted to \emph{convex geometries}, combinatorial objects that appear in various areas of mathematics in different disguise, see survey by B.~Monjardet \cite {Mo85}.  Treated in discrete mathematics as (finite) set systems, similar to \emph{matroids}, and dual to \emph{antimatroids}, convex geometries were also recognized as lattices with unique irredundant decompositions, see R.~P.~Dilworth \cite{D2}, but also closure systems embracing the \emph{anti-exchange property}. In geometry, these systems epitomized the concept of convexity in its discrete form. In recent years, more development was done for the infinite convex geometries and anti-matroids, see \cite{AGT03, AN16, AP11, A09, S15, MS00}. In particular, this development can be seen as the natural generalization of geometrical representations for the closure systems with the anti-exchange axiom associated with the convex bodies in Euclidean spaces, such treatment given, in particular, in \cite{AGT03,AP11}.   

In current paper we consider the generalization of the classical result of P.~H.~Edelman and 
R.~E.~Jamison \cite{EJ85}, which gives representation of an arbitrary finite convex geometry as a join of ``elementary" convex geometries that can be defined on the same base set. Such elementary sub-geometries can be extracted from original geometry by considering the maximal chains connecting the smallest and largest elements of the geometry. At the same time, such chains define a particular ordering on the base set of the convex geometry, which are called \emph{compatible} orderings. Many important examples of convex geometries in infinite case satisfy the \emph{finitary condition} for their closure operator, equivalently, as lattices they are \emph{algebraic}. It turns out that under these additional condition the Edelman-Jamison representation can be generalized smoothly, and we had it in notes for a number of years, until the publication of N. Wahl \cite{W01} came to our attention. 

While the current paper revisits the main topic of \cite{W01}, it contains mostly new results, which also go beyond just representation, and establish important properties of convex geometries in algebraic case. The only borrowed result is Lemma \ref{max-chains}, which is proved in \cite[Theorem 2]{W01}. In particular, our Lemma \ref{Xordering} and Theorem \ref{v-rep} seem to navigate proper sail from the  finite case into realm of algebraic. 

Preliminaries on convex geometries and algebraic closure operators are given in sections \ref{ConGeo} and \ref{AlgCloOper}.

\section{Convex geometries}\label{ConGeo}
Let $X$ be a non-empty set, and $\op{Pow} X$ be the set of all subsets of $X$. We know that with respect to the order relation $\subseteq$, $\op{Pow} X$ has a structure of a complete boolean lattice. In this paper, we will be interested in considering algebraic closure operators, or some of their generalizations.

\begin{df}\label{algCO}
A mapping $\phi : \op{Pow} X \longrightarrow \op{Pow} X$ is called an \emph{algebraic (or finitary) closure operator} on set $X$, if for all $A,B\subseteq X$
\begin{itemize}
\item[(1)] $A \subseteq \phi(A)$;
\item[(2)] if $A\subseteq B$, then $\phi(A)\subseteq \phi(B)$;
\item[(3)] $\phi(\phi(A))=\phi(A)$;
\item[(4)] $\phi(A)=\bigcup \{\phi(B): B\subseteq A, |B|<\omega\}$.
\end{itemize}
\end{df}
While the first three properties say that $\phi$ is a closure operator on set $X$, the last property indicates that closures of arbitrary sets are fully defined by closures of their finite subsets.\\
Apparently, in case of finite $X$, the last property trivially holds, so it only needs to be considered for infinite $X$.

Now we introduce the convex geometries.
A pair $(X,\phi)$ will be called \emph{a closure space}, if $\phi$ is a closure operator on $X$; also, $\phi$ is \emph{zero-closure} operator, if $\phi(\emptyset)=\emptyset$.

\begin{df}\label{CG}\cite{AGT03}
A zero-closure space $(X, \phi)$ satisfies the \emph{anti-exchange
property} if the following statement holds,
 \begin{equation}
 \begin{aligned}
 x\in\phi(A\cup\{y\})\text{ and }x\notin A
 \text{ imply that }y\notin\phi(A\cup\{x\})\\
 \text{ for all }x\neq y\text{ in }X\text{ and all closed }A\subseteq X.
 \end{aligned}
 \tag{AEP}
 \end{equation}
We then say that $(X,\phi)$ is a \emph{convex geometry}.
\end{df}

Often convex geometries were treated as \emph{set systems}, i.e. pairs $(X, \mathcal {F})$, where $\mathcal{F} \subseteq  \op{Pow} X$ satisfies the following properties:

\begin{itemize}
\item[(1)] $\emptyset, X \in \mathcal{F}$;
\item[(2)] $\mathcal{F}$ is stable under intersection, i.e. $A,B \in \mathcal{F}$ implies $A\cap B \in \mathcal{F}$;
\item[(3)] for every $Y\in \mathcal{F}$, such that $Y\not = X$, there exists $x \in X\setminus Y$ such that $Y\cup\{x\}\in \mathcal{F}$.
\end{itemize}

This latter description of finite convex geometries is equivalent to one given in Definition \ref{CG}, if one treats the elements of $\mathcal{F}$ as the family $\op{Cld} (X,\phi)$ of closed subsets of $X$ with respect to closure operator $\phi$. Namely, $\mathcal{F}=\op{Cld} (X,\phi)=\{Y\subseteq X: \phi(Y)=Y\}$. Note that properties $(1),(2)$ characterize the closed subsets of arbitrary zero-closure operator, while property $(3)$ is responsible to reflect the anti-exchange property of closure operator.

Another standard observation is that $(\op{Cld} (X,\phi), \cap)$ is a (lower) semilattice, which is a subsemilattice of $(\op{Pow} X, \cap)$. All (lower) subsemilattices of $(\op{Pow} X, \cap)$ form a partially ordered set with respect to containment order $\subseteq$, this partially ordered set is in fact a lattice and denoted $\op{Sg}_{\bigwedge} (\op{Pow} X)$. When one passes from finite to arbitrary $X$, the only change one needs to make is to consider families $\mathcal{F} \subseteq  \op{Pow} X$ which are stable under \emph{arbitrary} intersections, thus, $(\mathcal{F}, \bigcap)$ becomes \emph{complete} semilattice and an element in $\op{Sg}_{\bigwedge} (\op{Pow} X)$, the latter should be treated as the lattice of \emph{complete} (lower) subsemilatticies of  $(\op{Pow} X, \cap)$.

Consider now fixed (finite) set $X$ and some $\mathcal{F,G}\in \op{Sg}_{\bigwedge} (\op{Pow} X)$.  The following provides the description of the join of $\mathcal{F,G}$ in lattice $\op{Sg}_{\bigwedge} (\op{Pow} X)$:
\[ 
\mathcal{F} \vee \mathcal{G}= \{ F\cap G \in \op{Pow} X: F \in \mathcal{F}, G \in \mathcal{G}\}.
\]

One of key observations about the formation of different convex geometries on the shared base set $X$ is the following result from P.~H.~Edelman \cite[Th.2.2]{E80}.

\begin{prop} If  $(X,\mathcal{F})$ and $(X,\mathcal{G})$ are two convex geometries on $X$, then $(X, \mathcal{F}\vee\mathcal{G})$ is also a convex geometry on $X$.
\end{prop}

The next important observation about finite convex geometries is that every convex geometry $(X, \mathcal{F})$ with $X=\{x_1,\dots, x_n\}$ contains at least one subfamily $\mathcal{F}^\sigma \subseteq \mathcal{F}$ of the form $\mathcal{F}^\sigma=\{\emptyset, \{x_{\sigma(1)}\}, \{x_{\sigma(1)}, x_{\sigma(2)}\}, \dots , \{x_{\sigma(1)}, x_{\sigma(2)}, \dots, x_{\sigma(n-1)}\}, X\}$, for some permutation $\sigma$ on $\{1,2,\dots, n\}$. Indeed, this follows from property (3) of convex geometry as the set system. Note that $(X, \mathcal{F}^\sigma)$ is itself a convex geometry on $X$, which is called \emph{monotone alignment} in \cite{EJ85}. It is associated with the particular total order $\prec$ on $X$: $x_{\sigma(1)} \prec x_{\sigma(2)}\prec \dots \prec x_{\sigma(n-1)} \prec x_{\sigma(n)}$, which is naturally associated with subfamily $\mathcal{F}^\sigma$. Namely, elements of $\mathcal{F}^\sigma$ are down-sets of totally ordered set $(X,\prec)$, and the convex geometry of down-sets of this $X$ will be denoted $\id (X,\prec)$. Since $\mathcal{F}^\sigma \subseteq \mathcal{F}$, such total  ordering on $X$ is called \emph{compatible} with the given convex geometry. It is easy to observe that $\F^\sigma$ forms the longest chain in $\F$. Moreover, all the longest chains in $\F$ are all of this form and have the length $|X|$.

In the following result of P.~H.~Edelman and R.~E.~ Jamison convex geometries of down-sets of some total orderings on $X$ play the role of ``elementary" convex geometries that generate any given convex geometry on $X$.

\begin{thm}\cite[Th. 5.2]{EJ85}\label{EJ} Let $G=(X,\mathcal{F})$ be a closure system on finite set $X$. The following are equivalent:
\begin{itemize}
\item[(1)] $G$ is a convex geometry;
\item[(2)] $G = \bigvee \{\id(X,\prec_i): i \in I\}$, where $\{\prec_i: i \in I\}$ is set of compatible total orderings on set $X$.
\end{itemize}
\end{thm}

We will consider the generalization of this result in section \ref{Gen}.
 
\section{Algebraic closure operators}\label{AlgCloOper}

As we noticed before, every closure operator $\phi$ on set $X$ is uniquely associated with the family $\op{Cld} (X,\phi) \subseteq \op{Pow} X$ of its closed subsets: $Y \in \op{Cld} (X,\phi)$ iff $\phi(Y)=Y$. We would like to identify the families $\mathcal{C}\subseteq \op{Pow} X$ that are represented as $\op{Cld} (X,\phi)$, for some \emph{algebraic} closure operator $\phi$ on $X$.

We recall that subfamily $\mathcal{F} \subseteq \op{Pow} X$ is called \emph{algebraic}, if
\begin{itemize}
\item [(i)] $\F$ is stable under arbitrary intersections, i.e. $\bigcap X_i \in \mathcal{F}$, for any $X_i \in \mathcal{F}$, $i \in I$;
\item[(ii)] $\bigcup X_i \in \mathcal{F}$, for any non-empty up-directed family $X_i \in \mathcal{F}$, $i \in I$.
\end{itemize}

Here the family $X_i$, $i\in I$, of elements $\op{Pow} X$ is called \emph{up-directed}, if for any $X_i,X_j$ there exists another member of the family $X_k$ such that $X_i\cup X_j\subseteq X_k$.
Item (i) allows empty family, for which $\bigcap \emptyset = X$, thus, $X \in \mathcal{F}$, for every algebraic $\mathcal{F}$. The following statement represents the common knowledge, the proof may be checked, for example, in \cite{BS81}. 

\begin{lm}\label{alg} Family $\mathcal{F}\subseteq \op{Pow} X$ is represented as $\op{Cld} (X,\phi)$, for some algebraic closure operator $\phi$ iff $\mathcal{F}$ is an algebraic subset of boolean lattice $\op{Pow} X$.
\end{lm}

Note that lattices $\op{Cld} (X,\phi)$ formed through algebraic closure operators, are also called \emph{algebraic}, and they are characterized as complete lattices where each element is a join of compact elements: those would be exactly the closures of finite subsets. Check more details in \cite[Ch.1, sections 4,5]{BS81}. 

As in G. Gr\"atzer \cite[Lemma 28]{GLTF}, one would observe that Lemma \ref{alg} actually establishes Galois correspondence between algebraic closure operators on set $X$ and algebraic subsets of $\op{Pow} X$. Firstly, two mappings defined in Lemma \ref{alg} are the inverses of each other, and secondly, each of them reverses a natural order, defined below, on the set of all closure operators and on the set of all algebraic subsets.

\begin{df}\label{op order}
Given two algebraic closure operators $\Delta$ and $\phi$ on set $X$, we set $\Delta \leq \phi$ iff $\Delta(Y)\subseteq \phi(Y)$, for every $Y\subseteq X$. The partially ordered set of all algebraic closure operators is denoted $(\op{AClo} X,\leq)$.

Given two algebraic subsets $\mathcal{G}$ and $\mathcal{F}$, we define $\mathcal{G} \leq\mathcal{F}$ iff
$\mathcal{G} \subseteq\mathcal{F}$. The partially ordered set of all algebraic subsets of $\op{Pow} X$ is denoted
$(\op{S_p} (\op{Pow} X), \leq)$.
\end{df}

We note that $\Delta \leq \phi$, for closure operators $\Delta, \phi$ implies that every $\phi$-closed set is $\Delta$-closed, and that the lattice of closed sets of $\Delta$ will include the lattice of closed sets of $\phi$ as a lower subsemilattice.

\begin{thm}\label{AClo rep}
Both $(\op{AClo} X,\leq)$ and $(\op{S_p} (\op{Pow} X), \leq)$ are complete lattices, moreover, $(\op{AClo} X,\leq) \cong^\delta (\op{S_p} (\op{Pow} X), \leq)$ as complete lattices.
\end{thm}

We also observe the relation between $\op{S_p} (\op{Pow} X)$ and $\op{Sg}_{\bigwedge} (\op{Pow} X)$. The latter notation is for the lattice of complete meet-subsemilattices with $1$ of $\op{Pow} X$. Each element $\mathcal{F} \in \op{Sg}_{\bigwedge} (\op{Pow} X)$ represents the family of closed sets of some closure operator on $X$.
The following result is proved in V. Gorbunov \cite[Theorem 6.9]{G98}, see also a slightly stronger version in K. Adaricheva \cite[Theorem 3.2]{A11} that avoids direct reference to quasi-varieties. For any family $\mathcal{F}\subseteq \op{Pow} X$, we denote by $\op{Sg}_{\bigwedge}(\mathcal{F})$ the family generated by $\mathcal{F}$ and closed under arbitrary intersections, while $\widetilde{F}$ is the family closed under the unions of up-directed subfamilies in $\mathcal{F}$.

\begin{thm}\label{Sp in Sub}
Let $X$ be an arbitrary set.
\begin{itemize} 
\item[(1)] $\op{S_p} (\op{Pow} X)$ is a complete $\bigwedge$-subseimilattice and $\vee$-subsemilattice of $\op{Sg}_{\bigwedge} (\op{Pow} X)$.
\item[(2)] For any $\mathcal{F}\subseteq \op{Pow} X$, the minimal element $\mathcal{F}^* \in \op{S_p} (\op{Pow} X)$, containing $\mathcal{F}$, can be obtained as $\mathcal{F}^*=\widetilde{\op{Sg}_{\bigwedge}( \mathcal{F})}$. 
\end{itemize}
\end{thm}

The implication from this Theorem is that the join of arbitrary collection of algebraic subsets is generally larger than closing it by arbitrary intersections: it requires to add the unions of up-directed subfamilies. Therefore, $\op{S_p} (\op{Pow} X)$ does not form a complete $\bigvee$-subsemilattice in $\op{Sg}_{\bigwedge} (\op{Pow} X)$.

\section{ Algebraic convex geometries}

It is well-known that a convex geometry on a finite set $X$ is always a \emph{standard} closure system, which means that $\phi(\{x\})\setminus \{x\}$ is closed, for every $x \in X$. This observation can be generalized to algebraic convex geometries.

\begin{prop}\label{alg geom stan}
Let\/ $(X, \phi)$ be an algebraic convex geometry.  Then
$\phi(\{x\}) \setminus \{x\}$ is closed for every $x \in X$. 
\end{prop}
\begin{proof}
Suppose $\phi(\{x\})\setminus \{x\} = P \not = \phi(P)$, for some $x \in X$. Since $\phi(P)=\phi(\{x\})$ and $\phi$ is an algebraic operator, there exists a finite subset $P'\subseteq P$ such that $\phi(P')=\phi(\{x\})$. We may assume that $P'$ is minimal with this property. Note that $P'\not = \emptyset$ due to definition of convex geometry. Then for every $p\in P'$ we have $\phi(P'\setminus \{p\}) \subset \phi(\{x\})$ and $p \not \in \phi(P'\setminus \{p\})$. Denoting $A=\phi(P'\setminus \{p\})$, we have $x,p \not \in A$ and $x \in \phi(A\cup \{p\})$, $p \in  \phi(A\cup \{x\})$, which contradict (AEP).
\end{proof}

The next statement immediately follows from properties of standard closure operator.

\begin{cor}
In every algebraic convex geometry, for every $x \in X$, $\phi(\{x\})$ is a completely join irreducible element of $\Cld(X,\phi)$.
\end{cor}

The statement of Proposition \ref{alg geom stan} is no longer true in non-algebraic convex geometries.
\begin{exm}
\end{exm}
Let $X=N\cup\{x\}$, for some countable set $N$. Then define closure operator $\phi$ as follows:\\
\[
\phi(Y)=
\begin{cases}
X &\text{if } Y \text{ is co-finite or contains } x;\\
Y &\text{otherwise}.
\end{cases}
\] 
It is easy to verify that $\phi$ satisfies (AEP), thus, it is a convex geometry. We also observe that $\phi(\{x\}) = X$, and $N=X\setminus \{x\}$ is not closed.
\vspace{0.5cm}

We note that every standard closure system is \emph{reduced} and \emph{zero-closure}. The first property means that the closures of different singletons must be different,

check, for example, \cite[Section 2]{ANR13}. Zero-closure property is adopted in the definition of arbitrary convex geometry, but the closure operator of convex geometry might not be reduced in non-algebraic case.

In K.~Adaricheva and M.~Pouzet \cite{AP11}, there were some further observations on even wider class of convex geometries, those with \emph{weakly atomic} lattice of closed sets.
The latter property means that every interval $[a,b]$ has a pair of elements $c,d$ that form a cover: $c\prec d$. We call such elements \emph{a covering pair}. It is well-known that every algebraic lattice is weakly atomic, see, for example, \cite{G03}. Thus, the following statement is a form of generalization from the algebraic to weakly atomic case. We recall that a (complete) lattice is called \emph{spatial}, if every element is an (infinite) join of completely join irreducible elements.

\begin{lm}\label{weakat}
Suppose convex geometry $G=(X, \phi)$ satisfies the property that every interval $[A,B]\subseteq L=\op{Cld}(X,\phi)$ of closed sets has a covering pair:
$A\subseteq A'\prec B'\subseteq B$. Then $L$ is spatial. 
\end{lm}

The part of the proof of Lemma \ref{weakat} was to show that if two closed sets of any convex geometry form a covering pair $X_1\prec X_2$, then $|X_2\setminus X_1|=1$. The next statement shows such a property of closed sets in algebraic closure systems holds only in convex geometries. This generalizes the result of S.P.~Avann \cite[Th. 5.2]{Av61} for finite convex geometries.

\begin{prop}\label{covCG}
For algebraic closure system $G=(X,\phi)$, the following are equivalent:
\begin{itemize}
\item[(1)] $G$ is a convex geometry;
\item[(2)] If $X_1\prec X_2$ in $\Cld (X, \phi)$, then $|X_2\setminus X_1|=1$.
\end{itemize}
\end{prop}
\begin{proof} That (1) implies (2) follows from the proof of Lemma \ref{weakat}, and we include it here for completeness. Note that assumption about algebraicity of the operator is not needed here.

Indeed, let $c=X_1=\phi(X_1)\prec d=X_2=\phi(X_2)$ be a covering pair in  $\Cld (X, \phi)$. Pick any $x \in X_2\setminus X_1$. Then $X_2=\phi(X_1 \cup \{x\})$. If there is another $y \in X_2\setminus X_1$, $y\not = x$, then $y \in \phi(X_1\cup x)$ implies $x \not \in \phi(X_1\cup y)$. Hence $X_1 < \phi(X_1\cup \{y\}) < \phi(X_1\cup \{x\})=X_2$, a contradiction to $X_1\prec X_2$.

Now assume (2) and consider $x\not = y \in X\setminus A$, for some $A=\phi(A)$, such that $x\in \phi(A\cup \{y\})=A_y$. Take a maximal chain in $\mathcal{C} \subseteq [A,A_y]\setminus\{A_y\}$, and consider $\bigvee \mathcal{C}=C_0 \in \Cld(X,\phi)$. By the algebraicity of the operator, we get $C_0=\bigvee \mathcal{C} = \bigcup \mathcal{C}$. If $C_0 =A_y$, then we get a contradiction between $\phi(\{y\}) \leq \bigcup \mathcal{C}$ and $y \not \in C$, for every $C\in \mathcal{C}$. Hence, $C_0\prec A_y$, and, by assumption, $|A_y\setminus C_0|=1$. This implies $A_y=C_0\cup \{y\}$. Moreover, every maximal chain in $[A,A_y]$ contains $C_0$, therefore, $C_0$ is a unique lower cover of $A_y$ in interval $[A,A_y]$.

If we assume that $y \in \phi(A\cup \{x\})=A_x$, then $A_y=A_x$, so the same argument as above leads to $A_y=C_0\cup\{x\}$, a contradiction. Hence,  $A_x \subseteq C_0$ and $y \not \in A_x$. This implies (AEP).
\end{proof}

\section{Generalization of Edelman-Jamison Theorem for convex geometries on a fixed set}\label{Gen}

In this section we want to consider all possible \emph{algebraic} convex geometries defined on a given set $X$. 
The focus of this section is the generalization of representation of convex geometry via compatible orderings given in Theorem \ref{EJ} of  P.~H.~Edelman and R~E.~Jamison. There were further efforts, for the case of algebraic convex geometries, see N. Wahl \cite{W01}. We will fine-tune latter results and provide some new observations.

If $\mathcal{G}$ is the family of closed sets of an algebraic convex geometry, then $\mathcal{G}$ is an element of 
$\op{S_p} (\op{Pow} X)$,  
as shown in Lemma \ref{alg}. We recall that $\op{S_p} (\op{Pow} X)$ itself is contained in $\op{Sg}_{\bigwedge} (\op{Pow} X)$, the latter representing all complete meet subsemiattices in $\op{Pow} X$, or the families of closed sets of closure operators on $X$. 

We will denote $\op{ACG} X\subseteq \op{S_p} (\op{Pow} X)$ the collection of all algebraic convex geometries defined on $X$, ordered by the containment order on their families of closed sets, see Definition \ref{op order}. 

Our first effort is to show that  $\op{ACG} X$ is a complete $\bigvee$-subsemilattice in $\op{S_p} (\op{Pow} X)$. For this, we will use Theorem \ref{Sp in Sub} that tells that obtaining the smallest family $\mathcal{F}^* \in \op{S_p} (\op{Pow} X)$, containing any given family $\mathcal{F}\subseteq \op{Pow} X$ can be done in two steps:
\begin{itemize}
\item build family $\op{Sg}_{\bigwedge}(\mathcal{F})$ adding arbitrary intersections of subfamilies in $\mathcal{F}$;
\item build family $\widetilde{\op{Sg}_{\bigwedge}(\mathcal{F})}$ adding the unions of non-empty up-directed subfamilies in 
$\op{Sg}_{\bigwedge}(\mathcal{F})$.
\end{itemize}

The following result shows that if we start from families $\mathcal{G}_i$, $i\in I$, of closed sets of some convex geometries on $X$, and $\mathcal{F}=\bigcup_{i\in I} \mathcal{G}_i$, then both $\op{Sg}_{\bigwedge}(\mathcal{F})$ and $\widetilde{\op{Sg}_{\bigwedge}(\mathcal{F})}$ will represent families of closed sets of another convex geometry on $X$. 

\begin{thm}\label{JoinSp} Let $X$ be an arbitrary set.
If $\mathcal{G}_i=\Cld(X,\phi_i)$, $i\in I$,
are families of closed sets of some algebraic convex geometries on $X$, then the smallest element $\mathcal{F}^*\in \op{S_p}(\op{Pow} X)$, containing $\mathcal{F}=\bigcup_{i\in I} \mathcal{G}_i$ is also a convex geometry.

\end{thm}
\begin{remark}
Using notation $\bigvee_{Sp}$ for the join operator in $\op{S_p} (\op{Pow} X)$, we could write  more compactly that $\mathcal{F}^* = \bigvee_{Sp} \{\mathcal{G}_i: i \in I\}$. Analogously, $\bigvee_{Sg}$ will stand for the join operator in  $\op{Sg}_{\bigwedge} (\op{Pow} X)$.  
\end{remark}
\begin{proof}
We split the argument into two parts.

First, we show that if $\mathcal{G}_i$, $i\in I$, are convex geometries on $X$, then
$\op{Sg}_{\bigwedge}(\bigcup_{i\in I}\mathcal{G}_i)= \bigvee_{Sg} \{\mathcal{G}_i: i \in I\}$ is also a convex geometry on $X$. Apparently, the latter family comprises the closed sets of some closure operator $\psi$, and, for every $A\subseteq X$, we have $\psi(A)=
\bigcap_{i\in I} \phi_i(A)$. Moreover, $\psi(\emptyset)=\emptyset$. Thus,  we only need to show that $\psi$ satisfies (AEP).

Take any $A=\psi(A), x,y\not \in A$ and $x \in \psi(A\cup \{y\})$. We claim that there exists $i\in I$ such that $x,y \not \in \phi_i(A)$. Indeed, suppose not, and $I=I_1\cup I_2$, where $x \in \phi_i(A)$, for $i\in I_1$, and $y \in \phi_j(A)$, for $j\in I_2$. Define closure operators $\tau$ and $\phi$ on $X$ as follows: $\tau(A)=\bigcap_{i\in I_1}\phi_i(A)$ and $\phi(A)=\bigcap_{j\in I_2}\phi_j(A)$, $A\subseteq X$. Apparently, $\psi(A)=\tau(A)\cap \phi(A)$ and
$x\in \tau(A)$, $y\in \phi(A)$. Then $x \in \psi(A\cup\{y\})\subseteq \phi(\phi(A)\cup\{y\})\subseteq \phi(\phi(A))=\phi(A)$, which implies $x \in \phi(A)\cap\tau(A)=\psi(A)$, a contradiction. 

Thus, we may assume that $x,y \not \in \phi_i(A)$, for some $i\in I$.
Since $x \in \phi_i(\phi_i(A)\cup \{y\})$, we apply (AEP) that holds for $\phi_i$ to conclude $y \not \in \phi_i(\phi_i(A)\cup \{x\})$. This implies $y \not \in \psi(A\cup\{x\})$, which is needed.

Secondly, consider $\mathcal{F}^*=\widetilde{\op{Sg}_{\bigwedge}(\mathcal{F})}$, assuming that $\op{Sg}_{\bigwedge}(\mathcal{F})$ represents the family of closed sets of convex geometry $(X, \psi)$. According to Theorem \ref{Sp in Sub}, $\mathcal{F}^*$ represents the family of closed sets of some algebraic closure operator $\rho$ on $X$. Apparently, $\rho(A)\subseteq \psi(A)$, for every $A\subseteq X$. We need to show that $\rho$ satisfies (AEP).

So take some $A=\rho(A)$, $x,y \not \in A$ and $x \in \rho(A\cup\{y\})$. If $A=\psi(A)$, then we use (AEP) for $\psi$ to conclude that $y \not \in \psi(A\cup \{x\})$. This implies $y \not \in \rho(A\cup\{x\})$.

Otherwise, $A=\bigcup_{i\in I} A_i$, for some up-directed family of $\psi$-closed sets $A_i$, $i\in I$.

Since $x,y \not \in A$, we have $x,y \not \in A_i$, for every $i\in I$. 

The sub-family $\psi(A_i\cup \{y\})$, $i \in I$, is up-directed in $\op{Sg}_{\bigwedge}\mathcal{F}$, hence,\\ $\bigcup_{i\in I}\psi((A_i\cup \{y\})\in \mathcal{F}^*$. Moreover, $A\cup \{y\} \subseteq \bigcup_{i\in I}\psi((A_i\cup \{y\})$. Therefore, $\rho(A\cup\{y\})\subseteq \bigcup_{i\in I}\psi((A_i\cup \{y\})$. This implies $x \in \psi((A_i\cup \{y\})$, for some $i\in I$. Pick any $j\in I$. Since the family $(A_i, i \in I)$ is up-directed, we can find another $k\in I$ such that $A_j\subseteq A_k$ and $x \in \psi(A_k\cup\{y\})$. Applying (AEP) that holds for $\psi$, we obtain $y \not \in \psi(A_k\cup\{x\})$, thus, also $y \not \in \psi(A_j\cup\{x  \})$. We conclude that $y \not \in \bigcup_{j\in I} \psi(A_j\cup\{x\})$, hence, also $y \not \in \rho(A\cup\{x\})$, which is due to $\rho(A\cup\{x\})\subseteq  \bigcup_{i\in I} \psi(A_i\cup\{x\})$.   

\end{proof}

\begin{cor}\label{ACG}
$\op{ACG} X$ is a complete $\bigvee$-subsemilattice in $\op{S_p}(\op{Pow} X)$.
\end{cor}

We note that while $\mathcal{G}\cap \mathcal{F}$ is a family of some algebraic closure operator, if $\mathcal{G}$ and $\mathcal{F}$ are such, it does not necessarily gives the family of closed sets of a convex geometry, when both $\mathcal{G}, \mathcal{F}$ are convex geometries,  i.e., $\op{ACG} X$ does not form a meet subsemillattice in $\op{S_p} (\op{Pow} X)$.
Indeed, if $X=\{1,2\}$, and $\mathcal{G}=\{\emptyset, \{1\}, X\}$,  $\mathcal{F}=\{\emptyset, \{2\}, X\}$, then both
$\mathcal{G},\mathcal{F}$ are convex geometries, while $\mathcal{G}\cap \mathcal{F}=\{\emptyset, X\}$ is not.

It was proved in \cite[Theorem 2.2]{EJ85} that every maximal chain of a finite convex geometry on set $X$ has the length equal to $|X|$. Equivalently, each maximal chain has $|X|$ covering pairs. Similar result holds in case of algebraic convex geometries.

\begin{lm}\label{Xordering} 
Let $G=(X,\phi)$ be an algebraic convex geometry. For a maximal chain $\mathcal{C}$ in $L=\op{Cld}(X,\phi)$, let $\mathcal{C}^*=\{D\in \mathcal{C}: D_*\prec D \text{ for some } D_*\in \mathcal{C}\}$.
Define a mapping $h_C:X\rto \mathcal{C}$ as
\[
 h_C(x)=\bigcap \{C'\in \mathcal{C}: x \in C'\}, x \in X. 
\]
Then $h_C$ is one-to-one and onto mapping from $X$ to $\mathcal{C}^*$. 
\end{lm}
\begin{proof}
First, we observe that $h_C$ is well-defined, due to maximality of chain $\mathcal{C}$. Secondly, if $C_0=h_C(x)$ does not have a lower cover in $\mathcal{C}$, then $C_0=\bigvee_L \{C''\in \mathcal{C}: x \not \in C''\}$. This contradicts to the fact that $\phi(x)\subseteq C_0$ is a compact element of $L$. Therefore, $C_0$ has a lower cover $C_*\prec C_0$ in $\mathcal{C}$, and $C_0=C_*\cup \{x\}$, by Proposition \ref{covCG}. This implies that $h_C(x)\not = h_C(y)$, for $x\not = y$.

Finally, if $D_*\prec D$ is any covering pair from $\mathcal{C}$, then $D=D_*\cup\{t\}$, for some $t \in X$. Hence, $h_C(t)=D$, and $h_C$ is onto. 
\end{proof}

Our next goal will be to establish stronger connection between maximal chains of (algebraic) convex geometry and \emph{compatible ordering} of the base set of the geometry.

 There is natural way to define algebraic convex geometry for any partially ordered (in particular, total ordered) set $(X,\leq)$ that is generalization of \emph{downset alignment} (correspondingly, \emph{monotone alignment}) of \cite{EJ85}.

\begin{df}\label{ideal}
Given partially ordered set $(P,\leq)$, a pair $(P, \phi)$, where $\phi(Q)=\downarrow Q= \{p\in P: p\leq q \text{ for some } q \in Q\}$, is called an \emph{ideal} closure system and its lattice closed sets is denoted $\op{Id}(P,\leq)$.
\end{df}

Evidently, the lattice of closed sets of $\op{Id}(P,\leq)$ is algebraic. It is also straightforward to check that operator $\phi$ in Definition \ref{ideal} satisfies (AEP), i.e., $(P, \phi)$ is a convex geometry. Moreover, if $(P,\leq)$ is a chain, then $\op{Id}(P,\leq)$ coincides with the lattice of ideals of this chain, where the latter is treated as a lattice. 

\begin{df}\label{ordering}
Given closure system $(X,\phi)$, the total ordering $\leq$ of the base set $X$ is called \emph{compatible} with the system, if $\op{Id}(X,\leq) \subseteq \Cld(X,\phi)$.
\end{df}

The following statement is a part of \cite[Theorem 1]{W01}, but without any assumptions on the closure system.

\begin{lm} If $(X,\leq)$ is a total ordering compatible with closure system $(X,\phi)$, then $\op{Id}(X,\leq)$ is a maximal chain in $\Cld(X,\phi)$.
\end{lm}
\begin{proof}
Suppose the compatible ordering gives the chain $\op{Id}(X,\leq)$ in $\Cld(X,\phi)$, which is not maximal, i.e. there exists closed set $T$ such that $\{T\}\cup \op{Id}(X,\leq)$ is a chain as well. Since $\op{Id}(X,\leq)$ is stable under arbitrary joins and meets, there exists $C_1,C_2 \in \op{Id}(X,\leq)$ such that they form a covering pair in $\op{Id}(X,\leq)$, and $C_1\subset T\subset C_2$. If $C_2\setminus C_1$ has two different elements $x_1,x_2$, then, assuming that $x_1\leq x_2$ in $(X,\leq)$, we obtain $C_1\subset \downarrow x_1 \subset C_2$, so that $C_1\subset C_2$ cannot be a covering pair in $\op{Id}(X,\leq)$. Hence, $C_2=C_1\cup \{x\}$, for some $x \in X$, and $T=C_1$ or $T=C_2$.
\end{proof}

Inverse statement is also true, under additional assumption on closure system.

\begin{lm}\label{max-chains}\cite{W01} If $(X,\phi)$ is an algebraic convex geometry, then, for every maximal chain $\mathcal{C} \subseteq \Cld (X,\phi)$ there exists the total ordering $\leq_C$ on $X$ such that $\op{Id} (X,\phi)=\mathcal{C}$.
\end{lm}

Indeed, the total ordering $\leq_C$ can be defined using mapping $h_C$ from Lemma \ref{Xordering}:
$x\leq_C y$ iff $h_C(x)\subseteq h_C(y)$.

\begin{cor}
The minimal elements of $\op{ACG} X$ are $\op{Id}(X,\leq)$, where $\leq$ ranges over all possible total orders on $X$.
\end{cor}

The following result extends P.~H.~Edelman and R.~E.~Jamison \cite[Theorem 5.2]{EJ85} to the case of algebraic convex geometries. We note that join operator in the statement below differs from one defined in \cite[Theorem 3]{W01}. More precisely, we replace operator $\bigvee_{Sg}$ in the latter publication by $\bigvee_{Sp}$. Apparently, operator $\bigvee_{Sg}$ cannot be used in a property sufficient for \emph{algebraic} closure system, since it only produces a minimal closure system from the given joinands. The algebraicity may not be achieved as it was manifested in the example given in \cite{W01}.

\begin{thm}\label{v-rep} 
Let $G=(X,\phi)$ be a closure system. The following are equivalent:
\begin{itemize}
\item[(1)] $G$ is an algebraic convex geometry;
\item[(2)] $G=\bigvee_{Sp}\{\op{Id}(X,\leq_i): i\in I\}$, where $\{\leq_i: i \in I\}$ is some set of total orderings on set $X$.
\end{itemize}

\end{thm}
\begin{proof}
(2) implies (1) due to Theorem \ref{JoinSp}. In other direction, one can take as set $\{\leq_i: i \in I\}$ all compatible orderings of given convex geometry. According to Lemma \ref{max-chains}, all such orderings correspond to maximal chains in $\Cld(X,\phi)$. Since every $Y\in \Cld (X,\phi)$ belongs to some maximal chain $\mathcal{C}$, $Y$ will be in $\op{Id}(X,\leq_C)$ for the corresponding ordering $\leq_C$. With this choice of set of total orderings, we obtain
$G=\bigcup\{\op{Id}(X,\leq_i): i\in I\}=\bigvee_{Sp}\{\op{Id}(X,\leq_i): i\in I\}$.

\end{proof}

It will be an interesting direction of future studies to explore the possibility to represent algebraic convex geometry by the means of the minimal number of total orderings on its base set. We consider a few examples below.

\begin{exm}
\end{exm} 
Consider the convex geometry $G=(\mathbb{R}, \phi)$ of convex sets of the chain of real numbers $(\mathbb{R},\leq)$. Natural ordering of real numbers $\leq$ is compatible with $G$ via maximal chain $C_\leq = \{(-\infty,r): r \in \mathbb{R}\}$. Reversed ordering $\leq^r$ ($t_1 \leq^r t_2$ iff $t_2\leq t_1$) is also compatible via maximal chain 
$C_{\leq^r}=\{(r,\infty): r \in \mathbb{R}\}$. Apparently, $G= C_1\vee C_2=\op{Id}(\mathbb{R},\leq) \vee_{Sp} \op{Id}(\mathbb{R},\leq)$. 
Chains $C_1,C_2$ contain all completely meet irreducible elements of $\Cld(\mathbb{R},\phi)$, and operator $\vee_{Sp}$ in this case is reduced to taking \emph{finite intersections} of elements in $C_1\cup C_2$.

\begin{exm}
\end{exm}
Consider convex geometry $G=\Pow \mathbb{N}$, i.e. the closure system on the set of natural numbers $\mathbb{N}$, with identical closure operator. In order to represent this geometry by means of maximal chains, one can take one maximal chain $C_n$ per each meet irreducible element  $k_n=\mathbb{N}\setminus\{n\}$, so that $k_n \in C_n$. In this case $G=\bigvee_{Sp} \{C_n: n\geq 1\}$, and operator $\bigvee_{Sp}$ is reduced to taking arbitrary meets of elements from $\bigcup \{C_n: n \in \mathbb{N}\}$, i.e., it acts equivalently to $\bigvee_{Sg}$.

However, the number of chains in this representation may be reduced. For example, one can choose $C_2,C_3, ...$ in such a way that $\mathbb{N}\setminus\{1,2\}$ is in $C_2$, $\mathbb{N}\setminus\{1,3\}$ is in $C_3$ etc. Then $k_1=\mathbb{N}\setminus\{1\}$ can be represented as the union of sets, each of which is intersection of sets from $\bigcup \{C_n: n\geq 2\}$. In other words, $G=\bigvee_{Sp} \{C_n: n\geq 2\}$.

\vspace{0.5cm}
\emph{Acknowledgments.} The author would like to express the gratitude to J.B.Nation, who has been an inspirational source during all the years we worked on the chapters about join semidistributive lattices, for the volume of Lattice Theory: Special Topics and Applications series initiated by G.~Gr\"atzer. We are also grateful to Masataka Nakamura for the pointer to the paper of N. Wahl \cite{W01}.

\end{document}